\documentclass[12pt,a4paper,reqno]{amsart}
\usepackage{amsmath}
\usepackage{amsfonts}
\usepackage{amssymb}
\usepackage{amsthm}
\usepackage{enumerate}
\usepackage[all]{xy}
\usepackage{graphicx}
\usepackage{fullpage}

\theoremstyle{plain}

  \newtheorem{proposition}[]{Proposition}
  \newtheorem{lemma}[]{Lemma}
  \newtheorem{theorem}[]{Theorem}
  \newtheorem{corollary}[]{Corollary}
  
  \newtheorem{remark}[]{Remark}

\theoremstyle{definition}
    \newtheorem{definition}{Definition}

\title{Exploration processes and SLE$_6$}
\author{Jianping Jiang }
\address{NYU-ECNU Institute of Mathematical Sciences at NYU Shanghai, 3663 Zhongshan
Road North, Shanghai 200062, China. }
\email{jjiang@nyu.edu}

\begin{document}

\begin{abstract}
We define radial exploration processes from $a$ to $b$ and from $b$ to $a$ in a domain $D$ of hexagons where $a$ is a boundary point and $b$ is an interior point. We prove the reversibility: the time-reversal of the process from $b$ to $a$ has the same distribution as the process from $a$ to $b$. We show the scaling limit of such an exploration process is a radial SLE$_6$ in $D$. As a consequence, the distribution of the last hitting point with the boundary of any radial SLE$_6$ is harmonic measure. We also prove the scaling limit of a similar exploration process defined in the full complex plane $\mathbb{C}$ is a full-plane SLE$_6$. A by-product of these results is that the time-reversal of a radial SLE$_6$ trace after the last visit to the boundary is a full-plane SLE$_6$ trace up to the first visit of the boundary.

\end{abstract}

\maketitle

\section{Introduction}
The chordal exploration process for percolation was introduced by Schramm in a seminal paper \cite{Sch00}. In that paper, Schramm shows that if the scaling limit of the chordal exploration process exists and is conformally invariant, then it must be a chordal SLE$_\kappa$. The value $\kappa=6$ can be determined by either the locality property or the crossing probabilities, since SLE$_6$ is the only SLE$_\kappa$ that satisfies the locality property \cite{LSW01} and Cardy's formula \cite{Car92}.

Shortly after \cite{Sch00}, Smirnov \cite{Smi01} proved Cardy's formula for the critical site percolation on the triangular lattice. He also outlined a strategy for using the conformal invariance of crossing probabilities to prove the convergence of the chordal exploration process to a chordal SLE$_6$. Later, Camia and Newman \cite{CN07} presented a detailed and self-contained proof of this convergence based on Smirnov's strategy. Smirnov also outlined a different strategy in \cite{Smi07}. See Werner \cite{Wer07} for a detailed proof of this new strategy.

In section 4.3 of \cite{Wer07}, Werner defined a radial exploration process for percolation on a hexagonal lattice by concatenating a family of chordal exploration paths in a family of decreasing domains. Then he sketched a proof of the convergence of this radial exploration process to radial SLE$_6$. In this paper, we will define a different version of the radial exploration process, and then we will give a detailed proof of the same convergence.

In \cite{She09}, Sheffield defined an exploration path between a boundary point and another point on a hexagonal lattice domain. In \cite{Ken14}, Kennedy defined a smart kinetic self-avoiding walk (SKSAW) between two arbitrary points in any lattice domain. Similar definitions have appeared in the physics literature since the mid 1980's, see \cite{WT85} and \cite{KL85}. By considering several simple examples, one can see that none of these existing definitions of radial exploration processes satisfies the reversibility property. Our definition of radial exploration process is very similar to the exploration path in \cite{She09} and SKSAW in \cite{Ken14}, with a modification in order to get the reversibility property, which is the key to our proof of the properties of SLE$_6$ (see Corollary \ref{cor} and Theorem \ref{thm3} below).

Let $D$ be a simply connected domain in the complex plane with $a\in\partial D$ and $b\in D$, and let $D_{\delta}$ be the largest connected component of hexagons of $D\cap \mathcal{H}_{\delta}$ where $\mathcal{H}_{\delta}$ is the hexagonal lattice with mesh $\delta$. \textit{Mid-edges} of $\mathcal{H}_{\delta}$ are centers of edges of $\mathcal{H}_{\delta}$. In figure \ref{fig1}, one mid-edge is labeled by a small black square while one vertex of $\mathcal{H}_{\delta}$ is labeled by a circle. The idea of using mid-edges instead of vertices is motivated by \cite{DCS12}. This idea is the main difference between our definition and the definitions of radial exploration processes in \cite{Wer07}, \cite{She09} and \cite{Ken14}. Our definition is necessary to obtain the reversibility property. Let $a_{\delta}$ be a closest mid-edge to $a$ in the set of mid-edges outside of $D_{\delta}$ but within $\delta/2$ distance from the topological boundary of $D_{\delta}$, $\partial D_{\delta}$, and let $b_{\delta}$ be a closest mid-edge to $b$ in $D_{\delta}$. In case such closest mid-edges are not unique, one may choose an arbitrary one. 
\begin{definition}
Let $\gamma_{\delta}(0)=a_{\delta}$. In the first step, there are two mid-edges in $D_{\delta}$ within distance $\delta$ from $a_{\delta}$, each of them is picked with probability $1/2$ independently. We denote the picked mid-edge by $\gamma_{\delta}(1)$. In the $k$-th step ($k\geq 2$), there are at most two mid-edges (call them \textit{allowable}) that are within distance $\delta$ from $\gamma_{\delta}(k-1)$ and connected to $b_{\delta}$ in $D_{\delta}\setminus\gamma_{\delta}[0,k-1]$ (i.e., there exists a polygonal path from those mid-edges to $b_{\delta}$, contained in $D_{\delta}$, that does not cross $\gamma_{\delta}[0,k-1]$). We view $\gamma_{\delta}[0,k-1]$ as a continuous polygonal path (i.e., a continuous path using only edges of $\mathcal{H}_{\delta}$), and we also require that $\gamma_{\delta}$ evaluated at half-integers are the vertices of $\mathcal{H}_{\delta}$. A simple induction argument shows that there is always at least one such allowable mid-edge. We pick each of the allowable mid-edges with probability $1/2$ independently of all previous choices if there are two; we pick the allowable mid-edge if there is only one.  Denote the new picked mid-edge by $\gamma_{\delta}(k)$. We stop the process when $\gamma_{\delta}$ reaches $b_{\delta}$. The resulting polygonal path $\gamma_{\delta}$ is called the \textit{radial exploration process from $a_{\delta}$ to $b_{\delta}$ in $D_{\delta}$}.
\end{definition}
\begin{figure}
\centering
\includegraphics[height=2.5in,width=3.2in]{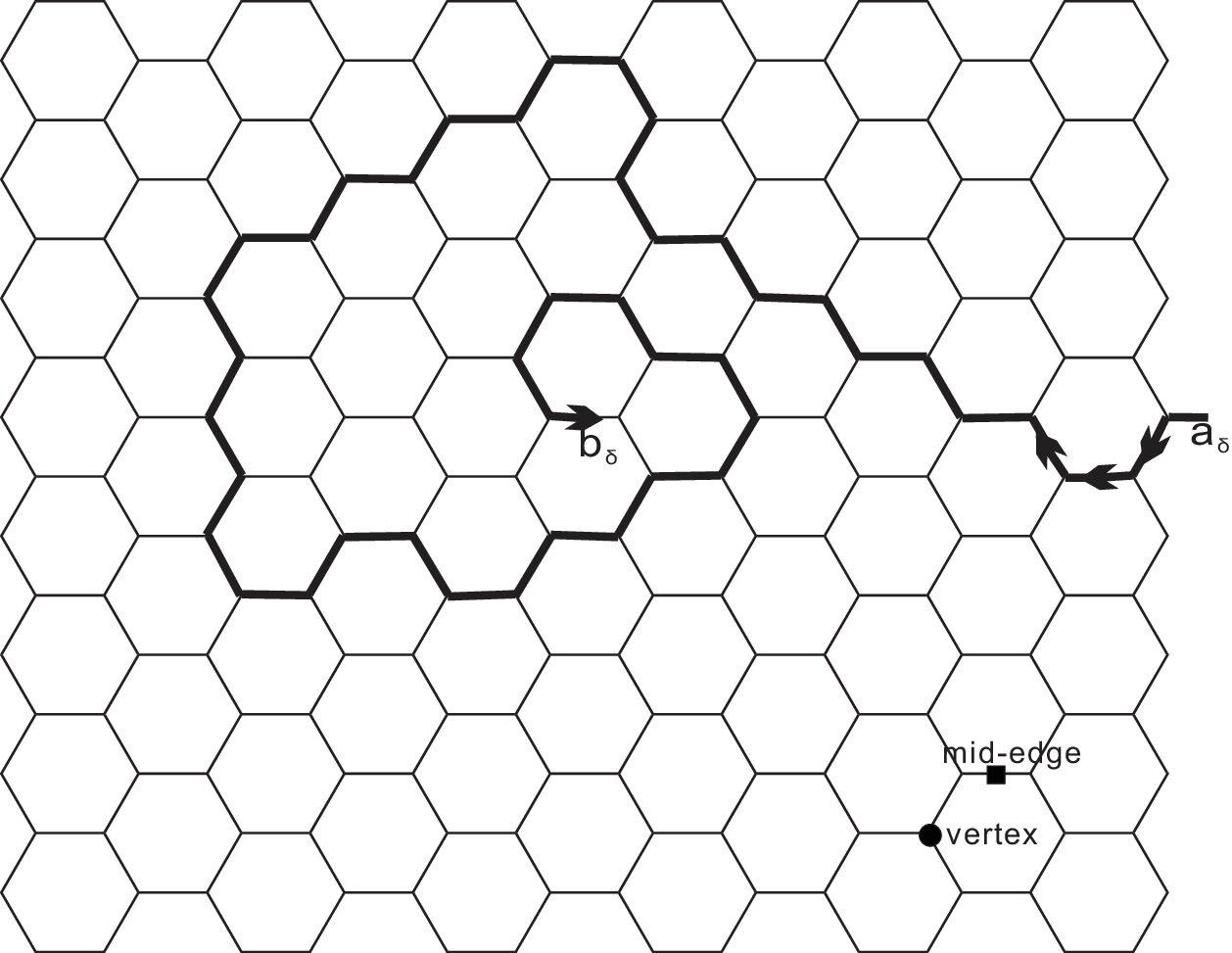}
\caption{The heavy path is a realization of the radial exploration process $\gamma_{\delta}$  from $a_{\delta}$ to $b_{\delta}$ in $D_{\delta}$. Here $D_{\delta}$ consists of all hexagons shown. $\gamma_{\delta}(1)$, $\gamma_{\delta}(2)$ and $\gamma_{\delta}(3)$ are labeled by arrows. }\label{fig1}
\end{figure}

We will see in Lemma \ref{lem3.1} that each walk $\gamma_{\delta}$ has weight $(1/2)^{l(\gamma_{\delta})}$ where $l(\gamma_{\delta})$ is the number of hexagons in $D_{\delta}$ sharing at least a half-edge with $\gamma_{\delta}$. This weight formula, $(1/2)^{l(\gamma_{\delta})}$, is the same as the weight formula for the chordal exploration process. We think this is one of the advantages of our definition of the radial exploration process. Another advantage is the reversibility that we will describe below.

We next define a radial exploration process (say, $\gamma_{\delta}^{\prime}$) from $b_{\delta}$ to $a_{\delta}$ in $D_{\delta}$. 
\begin{definition}
Let $\gamma_{\delta}^{\prime}(0)=b_{\delta}$. In the first step, there are exactly 4 mid-edges in $D_{\delta}$ (if $\delta$ is small enough) within distance $\delta$ from $b_{\delta}$, and each of them is picked with probability $1/4$ independently. We denote the picked mid-edge by $\gamma_{\delta}^{\prime}(1)$. Recall that once $a_{\delta}$ is fixed, $\gamma_{\delta}[0,1/2]$ is the unique half-edge starting at $a_{\delta}$ and connected to $D_{\delta}$.  In the $k$-th step ($k\geq 2$), there are at most two mid-edges (call them \textit{allowable}) that are within distance $\delta$ from $\gamma_{\delta}^{\prime}(k-1)$ and connected to $a_{\delta}$ in $D_{\delta}\cup\gamma_{\delta}[0,1/2]\setminus\gamma_{\delta}^{\prime}[0,k-1]$ (i.e., there exists a polygonal path from those mid-edges to $a_{\delta}$, contained in $D_{\delta}\cup\gamma_{\delta}[0,1/2]$, that does not cross $\gamma_{\delta}^{\prime}[0,k-1]$). We view $\gamma_{\delta}^{\prime}[0,k-1]$ as a continuous polygonal path, and we also require that $\gamma_{\delta}^{\prime}$ evaluated at half-integers are the vertices of $\mathcal{H}_{\delta}$. We pick each of the allowable mid-edges with probability $1/2$ independently if there are two; we pick the allowable mid-edge if there is only one. Denote the new picked mid-edge by $\gamma_{\delta}^{\prime}(k)$. We stop the process when $\gamma_{\delta}^{\prime}$ reaches $a_{\delta}$.  The resulting polygonal path $\gamma_{\delta}^{\prime}$ is called the \textit{radial exploration process from $b_{\delta}$ to $a_{\delta}$ in $D_{\delta}$}.
\end{definition}

Our first result is about the reversibility of radial exploration processes:
\begin{lemma}\label{lem1}
For any simply connected domain D, the radial exploration process from $a_{\delta}$ to $b_{\delta}$ in $D_{\delta}$ has the same distribution as the time-reversal of the radial exploration process from $b_{\delta}$ to $a_{\delta}$ in $D_{\delta}$.
\end{lemma}

Next, we will adopt the strategy outlined by Werner \cite{Wer07} and the techniques developed by Camia and Newman (\cite{CN06} and \cite{CN07}) to prove
\begin{theorem}\label{thm1}
Suppose $D$ is a Jordan domain. As $\delta\downarrow 0$, the radial exploration process in $D_{\delta}$ from $a_{\delta}$ to $b_{\delta}$ converges weakly to the radial SLE$_6$ in $D$ from $a$ to $b$.
\end{theorem}

\begin{remark}
Here, the distance between two continuous curves is the uniform metric on equivalence classes of curves modulo monotonic reparametrization, see \eqref{eq2.1} for the definition.
\end{remark}

As a result of Lemma \ref{lem1} and Theorem \ref{thm1}, we will prove
\begin{corollary}\label{cor}
Suppose $D$ is a Jordan domain that contains $0$. The distribution of the last hitting point of $\partial D$ of a radial SLE$_6$ in $D$ (aiming at $0$)  is the harmonic measure in $D$ started at $0$ regardless of the starting point of the radial SLE$_6$.
\end{corollary}

Analogously, we can define a full-plane exploration process in $\mathcal{H}_{\delta}$ from $0_{\delta}$ to $\infty$ where $0_{\delta}$ is a closest mid-edge to $0$ in $\mathcal{H}_{\delta}$. 
\begin{definition}
Let $\gamma_{\delta}(0)=0_{\delta}$. In the first step, there are exactly 4 mid-edges in $\mathcal{H}_{\delta}$ within distance $\delta$ from $0_{\delta}$, and each of them is picked with probability $1/4$ independently. We denote the picked mid-edge by $\gamma_{\delta}(1)$. In the $k$-th step ($k\geq 2$), there are at most two mid-edges (call them \textit{allowable}) that are within distance $\delta$ from $\gamma_{\delta}(k-1)$ and connected to $\infty$ in $\mathcal{H}_{\delta}\setminus\gamma_{\delta}[0,k-1]$. We pick each of the allowable mid-edges with probability $1/2$ independently if there are two; we pick the allowable mid-edge if there is only one. Denote the new picked mid-edge by $\gamma_{\delta}(k)$. The resulting polygonal path $\gamma_{\delta}[0,\infty)$ is called the \textit{full-plane exploration process from $0_{\delta}$ to $\infty$ in $\mathcal{H}_{\delta}$}.
\end{definition}
Because of its similarity with the radial exploration process, it is natural to conjecture that the scaling limit of this process is a full-plane SLE$_6$. See for example section 6.6 of \cite{Law05a} for the definition of full-plane SLE$_6$. Based on Corollary \ref{cor} and a similar convergence result as Theorem \ref{thm1} for unbounded domains, we will prove
\begin{theorem}\label{thm2}
As $\delta\downarrow 0$, the full-plane exploration process from $0_{\delta}$ to $\infty$ in $\mathcal{H}_{\delta}$ converges weakly to the full-plane SLE$_6$ in $\mathbb{C}$ from $0$ to $\infty$.
\end{theorem}
\begin{remark}
The distance between two continuous curves in $\mathbb{C}$ is defined in \eqref{eq2.2}.
\end{remark}

A direct consequence of Lemma \ref{lem1}, Theorem \ref{thm1} and Theorem \ref{thm2} is
\begin{theorem}\label{thm3}
Suppose $D$ is a Jordan domain that contains $0$. Then up to a time-change, the time-reversal of the radial SLE$_6$ in $D$ after the last hitting of $\partial D$ has the same distribution as the full-plane SLE$_6$ started at $0$ and stopped when it first hits $\partial D$.
\end{theorem}
\begin{remark}
As we mentioned, the reversibility of the radial exploration processes is essential for the proofs of Corollary \ref{cor} and Theorem \ref{thm3}. We do not see a way to prove these results for SLE$_6$ without using the exploration processes defined in this paper.
\end{remark}

The organization of the paper is as follows. In Section \ref{secpre} we define the metrics on curves and review the definition of SLE$_{\kappa}$. In Section \ref{secrev} we prove the reversibility of the radial exploration processes. Section \ref{sec4} is devoted to proving Theorem \ref{thm1}; while Section \ref{sec5} proves Theorems \ref{thm2} and \ref{thm3}.

\section{Preliminaries}\label{secpre}
\subsection{The space of curves}
We will identify the real plane $\mathbb{R}^2$ and the complex plane $\mathbb{C}$ in the usual way. A domain $D$ is a nonempty, connected and open subset of $\mathbb{C}$. A simply connected domain $D$ is said to be a Jordan domain if its boundary $\partial D$ is a Jordan curve (i.e., $
\partial D$ is a homeomorphism of the unit circle).

Let $D$ be a simply connected and bounded domain. Our space of curves in $\bar{D}$ is defined as the set of equivalence classes of continuous functions from $[0,1]$ to $\bar{D}$, modulo monotonic reparametrization. For any two continuous curves $\gamma_1$ and $\gamma_2$ , let $d(\cdot,\cdot)$ be the uniform metric on curves, i.e.,
\begin{equation}\label{eq2.1}
d(\gamma_1,\gamma_2):=\inf\sup_{t\in[0,1]}|\gamma_1(t)-\gamma_2(t)|,
\end{equation}
where the infimum is over all choices of parametrizations of $\gamma_1$ and $\gamma_2$  from the interval $[0,1]$. It is easy to check that $d(\cdot,\cdot)$ is a metric on the equivalent classes of curves. The space of continuous curves in $\bar{D}$ is complete and separable with respect to the metric \eqref{eq2.1}, but it is not necessarily compact, see \cite{AB99}.

Let $\hat{\mathbb{C}}:=\mathbb{C}\cup\{\infty\}$ be the Riemann sphere. For any two points $z_1,z_2\in\hat{\mathbb{C}}$, let $\Delta(\cdot,\cdot)$ be the spherical metric, i.e.,
$$\Delta(z_1,z_2):=\inf_{\gamma}\int_{\gamma}\frac{2|dz|}{1+|z|^2},$$
where $\gamma$ is any piecewise differentiable curve joining $z_1$ and $z_2$ in $\mathbb{C}$. This metric is equivalent to the Euclidean metric in any bounded regions. For any two continuous curves $\gamma_1$ and $\gamma_2$ in $\hat{\mathbb{C}}$, we define the distance between $\gamma_1$ and $\gamma_2$ as
\begin{equation}\label{eq2.2}
D(\gamma_1,\gamma_2):=\inf\sup_{t\in[0,1]}\Delta(\gamma_1(t),\gamma_{2}(t)),
\end{equation}
where the infimum is over all choices of parametrizations of $\gamma_1$ and $\gamma_2$  from the interval $[0,1]$. Here, again curves are regarded as equivalence classes, modulo monotonic reparametrization. The space of continuous curves in $\hat{\mathbb{C}}$, denoted by $\mathcal{K}$, is also complete and separable with respect to the metric \eqref{eq2.2} , but not compact. When we talk about weak convergence of measures on curves, we always mean with respect to the metric \eqref{eq2.1} or \eqref{eq2.2}. Let $\mathcal{B}_{\mathcal{K}}$ be the Borel $\sigma$-algebra on $\mathcal{K}$ induced by the metric \eqref{eq2.2}. Let $\mathcal{M}$ denote the set of probability measures on $\mathcal{K}$. For any $\mu,\nu\in \mathcal{M}$, the Prohorov metric $\rho$ on $\mathcal{M}$ defined by: $\rho(\mu,\nu)$ is the infimum of all $\epsilon>0$ such that for every $V\in\mathcal{B}_{\mathcal{K}}$,
$$\mu(V)\leq\nu(V^{\epsilon})+\epsilon,~~~\nu(V)\leq\mu(V^{\epsilon})+\epsilon,$$
where $V^{\epsilon}=\{\gamma:\inf_{\tilde{\gamma}\in V}D(\gamma,\tilde{\gamma})<\epsilon\}$. One property about $\rho$ we will use in this paper is: $\rho(\mu_n,\mu)\rightarrow 0$ as $n\rightarrow\infty$ if and only if $\mu_n$ converges weakly to $\mu$ as $n\rightarrow\infty$. See page 72 of \cite{Bil99} for a proof.

\subsection{SLE$_{\kappa}$}

\subsubsection{Chordal SLE$_{\kappa}$}
Let $(B_t)_{t\geq0}$ be a standard Brownian motion on $\mathbb{R}$ with $B_0=0$. Let $\kappa\geq 0$ and consider the solution to the chordal Loewner equation for the upper half plane $\mathbb{H}$,
$$\partial_tg_t(z)=\frac{2}{g_t(z)-\sqrt{\kappa}B_t},~~~g_0(z)=z, z\in\overline{\mathbb{H}}.$$
This is well defined as long as $g_t(z)-\sqrt{\kappa}B_t\neq 0$, i.e., for all $t<\tau(z)$, where $\tau(z):=\inf\{t\geq 0:g_t(z)-\sqrt{\kappa}B_t=0\}$. For each $t>0$, $g_t:\mathbb{H}\setminus K_t\rightarrow\mathbb{H}$ is a conformal map, where $K_t:=\{z\in\overline{\mathbb{H}}:\tau(z)\leq t\}$ is a compact subset of $\overline{\mathbb{H}}$ such that $\mathbb{H}\setminus K_t$ is simply connected. It is known (see \cite{RS05}) that $\gamma(t):=g_t^{-1}(\sqrt{\kappa}B_t)$ exists and continuous in $t$, and the curve $\gamma$ is called the \textit{trace} of chordal SLE$_{\kappa}$. It is also proven in the same paper that $\gamma$ is simple if and only if $\kappa\in[0,4]$.

Let $D$ be a simply connected domain and $a,c$ be distinct points on $\partial D$. Let $f:\mathbb{H}\rightarrow D$ be a conformal map with $f(0)=a$ and $f(\infty)=c$. If $\gamma$ is the chordal SLE$_{\kappa}$ trace in $\overline{\mathbb{H}}$, then $f\circ\gamma$ defines the chordal SLE$_{\kappa}$ trace from $a$ to $c$ in $\bar{D}$.

\subsubsection{Radial SLE$_{\kappa}$}
Radial SLE$_{\kappa}$ is defined similarly but using the radial Loewner equation
$$\partial_t g_t(z)=-g_t(z)\frac{g_t(z)+e^{i\sqrt{\kappa}B_t}}{g_t(z)-e^{i\sqrt{\kappa}B_t}},~~~g_0(z)=z, z\in\overline{\mathbb{D}},$$
where $\mathbb{D}:=\{z:|z|<1\}$ is the unit disk. The \textit{trace} $\gamma(t):=g_t^{-1}(\sqrt{\kappa}B_t)$ is now a continuous curve growing from $1$ to $0$ in $\overline{\mathbb{D}}$. See \cite{Law13} for the proof of $\lim_{t\rightarrow\infty}|\gamma(t)|=0$.

Let $D$ be a simply connected domain with $a\in\partial D$ and $b\in D$. Let $f:\mathbb{D}\rightarrow D$ be the conformal map with $f(1)=a$ and $f(0)=b$. If $\gamma$ is the radial SLE$_{\kappa}$ trace in $\overline{\mathbb{D}}$, then $f\circ\gamma$ defines the radial SLE$_{\kappa}$ trace from $a$ to $b$ in $\bar{D}$.

\subsubsection{Full-plane SLE$_{\kappa}$}
Let $(B_t^1)_{t\geq0}$ and $(B_t^2)_{t\geq0}$ be two independent Brownian motions starting at the origin, and $Y$ be uniformly distributed on $[0,2\pi/\sqrt{\kappa}]$ and independent of $B_t^1$ and $B_t^2$. Set $B_t=Y+B_t^1$ if $t\geq 0$, and $B_t=Y+B_{-t}^2$ if $t\leq 0$. The full-plane SLE$_{\kappa}$ (from $0$ to $\infty$) is the family of conformal maps $g_t$ satisfying
\begin{equation}\label{eq2.3}
\partial_t g_t(z)=-g_t(z)\frac{g_t(z)+e^{-iU_t}}{g_t(z)-e^{-iU_t}},
\end{equation}
where $U_t:=\sqrt{\kappa}B_t$ and the initial condition is $\lim_{t\rightarrow-\infty}e^tg_t(z)=z, z\in\mathbb{C}\setminus\{0\}$. Let $\gamma:(-\infty,\infty)\rightarrow\mathbb{C}$ with $\lim_{t\rightarrow-\infty}\gamma(t)=0$ and $\lim_{t\rightarrow\infty}\gamma(t)=\infty$ be the trace of full-plane SLE$_{\kappa}$. Then $g_t$ is the conformal transformation of the unbounded component of $\mathbb{C}\setminus\gamma[-\infty,t]$ onto $\mathbb{C}\setminus\overline{\mathbb{D}}$ with $g_t(z)\sim e^{-t}z$ as $z\rightarrow\infty$. We will see in section {\ref{sec5}}: conditioned on the $\gamma[-\infty,t]$ for any $t\in \mathbb{R}$, $\gamma[t,\infty]$ has the distribution of a radial SLE$_{\kappa}$ trace growing in $\hat{\mathbb{C}}\setminus\gamma[-\infty,t]$.

If $z,w$ are distinct points in $\mathbb{C}$, we can also define full-plane SLE$_{\kappa}$ connecting $z$ and $w$ by using a linear fractional transformation sending $0$ to $z$ and $\infty$ to $w$.

\section{Reversibility of radial exploration processes}\label{secrev}
In the introduction, we defined a radial exploration process $\gamma_{\delta}$ from $a_{\delta}$ to $b_{\delta}$ in $D_{\delta}$ and a radial exploration process $\gamma_{\delta}^{\prime}$ from $b_{\delta}$ to $a_{\delta}$ in $D_{\delta}$. Let us summarize those definitions here. Both $\gamma_{\delta}$ and $\gamma_{\delta}^{\prime}$ are simple polygonal paths (i.e., self-avoiding polygonal paths) defined step by step. At each step, a mid-edge adjacent to the tip of the exploration process is declared as an allowable mid-edge if it does not block the exploration process from reaching its target. The exploration process then chooses uniformly among the allowable mid-edges it has, independently at each step. Except that the first step of $\gamma_{\delta}^{\prime}$ has four allowable mid-edges, the number of allowable mid-edges is always either 1 or 2.

One can also define radial exploration processes using coloring algorithms. That is, the coloring algorithms generate paths with the same distribution as $\gamma_{\delta}$ and $\gamma_{\delta}^{\prime}$ which we defined in the introduction.

We define the coloring algorithm for $\gamma_{\delta}$ first. Let $\gamma_{\delta}(0)=a_{\delta}$, and let $\gamma_{\delta}[0,1/2]$ be the unique half-edge starting at $a_{\delta}$ and connected to $D_{\delta}$. In the first step, let $\xi$ be the hexagon in $D_{\delta}$ has the vertex $\gamma_{\delta}(1/2)$, $\xi$ is colored blue or yellow with probability $1/2$. We choose the left half-edge (with respect to $\gamma_{\delta}[0,1/2]$) if $\xi$ is blue, or right half-edge if $\xi$ is yellow. We denote the endpoint of the chosen half-edge (i.e., the mid-edge) by $\gamma_{\delta}(1)$. At the $k$-th step ($k\geq 2$), let $\xi$ be the hexagon centered at $\gamma_{\delta}(k-1)+3[\gamma_{\delta}(k-1)-\gamma_{\delta}(k-3/2)]$.
\begin{itemize}
\item
If $\xi$ has not been colored and is in $D_{\delta}$, we randomly color it blue or yellow with probability $1/2$, and we choose the left half-edge with respect to $\gamma_{\delta}[k-1,k-1/2]$ (note that $\gamma_{\delta}(k-1/2)$ is the other endpoint of the edge contains $\gamma_{\delta}[k-3/2,k-1]$) if $\xi$ is blue, or right half-edge if $\xi$ is yellow;
\item
    if $\xi$ has been colored or is in the complement of $D_{\delta}$, then we choose the half-edge adjacent to $\gamma_{\delta}[k-1,k-1/2]$ that is connected to $b_{\delta}$ in $D_{\delta}\setminus\gamma_{\delta}[0,k-1/2]$.
\end{itemize}
We denote the endpoint of the chosen half-edge (i.e., the mid-edge) by $\gamma_{\delta}(k)$. The algorithm stops when $\gamma_{\delta}$ reaches $b_{\delta}$. This coloring algorithm generates paths with the same distribution as the radial exploration precess from $a_{\delta}$ to $b_{\delta}$ simply because we color $\xi$ at the $k$-th step if and only if there are two allowable mid-edges at the $k$-th step.

\begin{figure*}
\begin{minipage}[c]{0.48\textwidth}
\centering
\includegraphics[height=1.9in,width=2.40in]{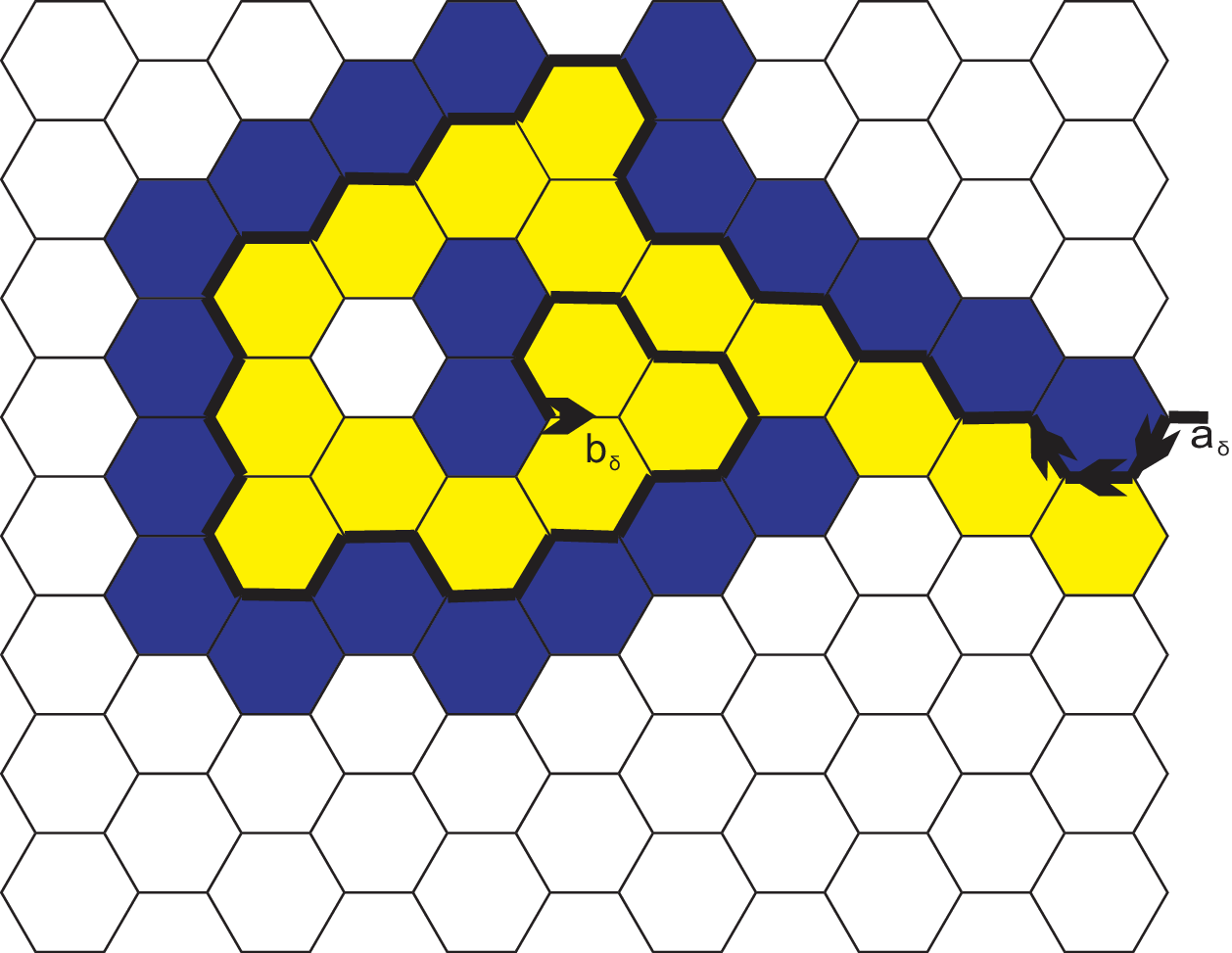}
\caption{$\gamma_{\delta}$. }\label{fig2}
\end{minipage}
\begin{minipage}[c]{0.48\textwidth}
\centering
\includegraphics[height=1.9in,width=2.40in]{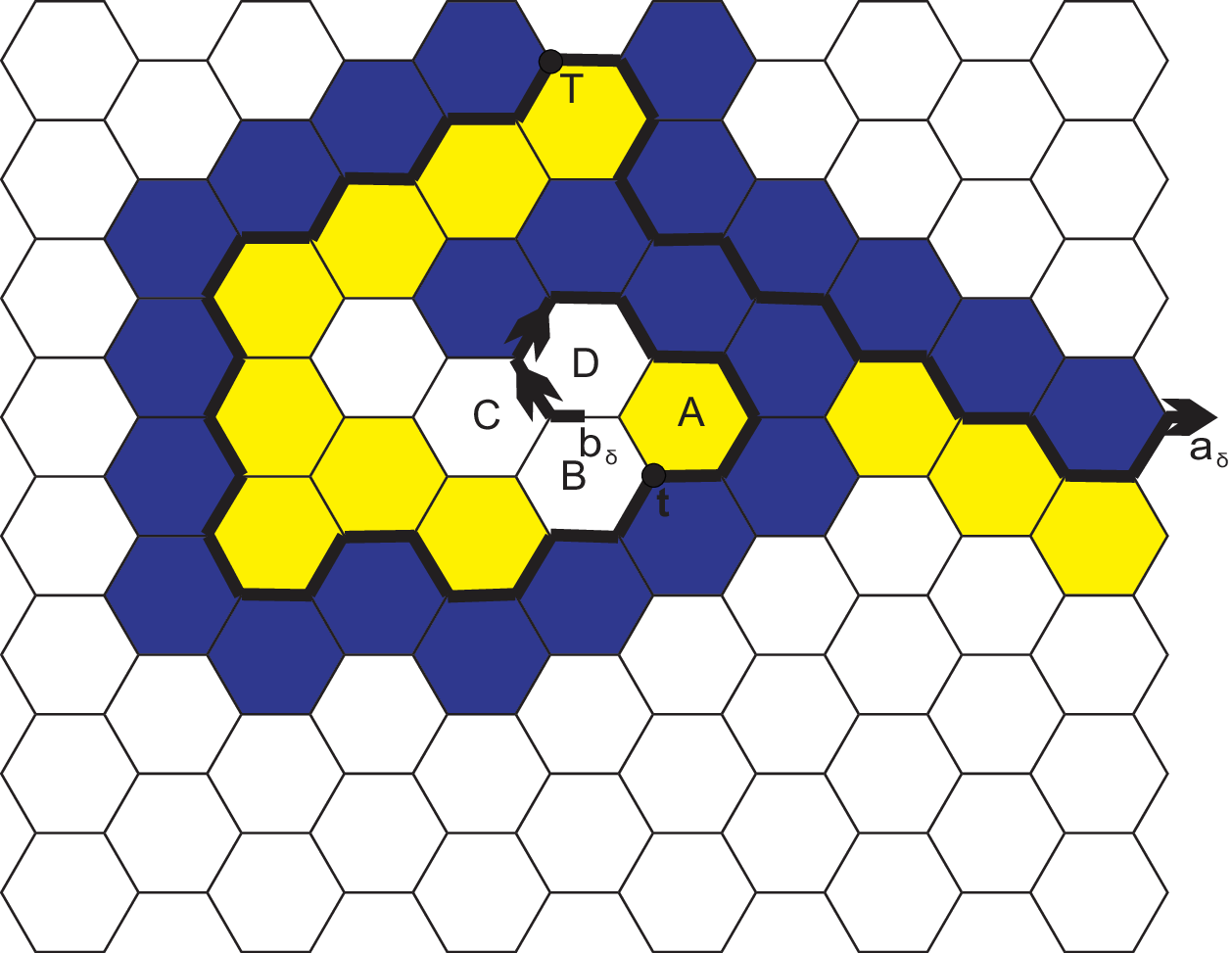}
\caption{$\gamma_{\delta}^{\prime}$.}\label{fig3}
\end{minipage}
\end{figure*}

The coloring algorithm for $\gamma_{\delta}^{\prime}$ is similar. We let $\gamma_{\delta}^{\prime}(0)=b_{\delta}$ and $\gamma_{\delta}^{\prime}(1)$ is picked with probability $1/4$ independently from the 4 mid-edges in $D_{\delta}$ with distance $\delta$ from $b_{\delta}$. Let $A, B, C, D$ be the four hexagons within distance $\delta$ from $b_{\delta}$ (the order does not matter). In the $k$-th step ($k\geq 2$), there are at most two mid-edges (call them allowable) that are within distance $\delta$ from $\gamma_{\delta}^{\prime}(k-1)$ and connected to $a_{\delta}$ in $D_{\delta}\cup\gamma_{\delta}[0,1/2]\setminus\gamma_{\delta}^{\prime}[0,k-1]$. Let $\xi$ be the hexagon centered at $\gamma_{\delta}^{\prime}(k-1)+3[\gamma_{\delta}^{\prime}(k-1)-\gamma_{\delta}^{\prime}(k-3/2)]$.
\begin{itemize}
\item If $\xi$ has not been colored and is in $D_{\delta}\setminus\{A,B,C,D\}$, we randomly color $\xi$ blue or yellow with probability $1/2$, and we choose the right half-edge (with respect to $\gamma_{\delta}^{\prime}[k-1,k-1/2]$) if $\xi$ is blue, or left half-edge if $\xi$ is yellow;
\item if $\xi$ is not in $D_{\delta}$ and there are two allowable mid-edges, this happens exactly when  $\gamma_{\delta}^{\prime}$ first hits $\partial D_{\delta}$, then we choose each of the allowable mid-edges with probability $1/2$ independently;
\item if $\xi$ is not in $D_{\delta}$ and there is only one allowable mid-edge then we choose this allowable mid-edge;
\item if $\xi$ has been colored and is in $D_{\delta}\setminus\{A,B,C,D\}$ then there is only one allowable mid-edges, and we choose this allowable mid-edge;
\item if $\xi$ is in $\{A, B, C, D\}$ and there are two allowable mid-edges, we randomly color $\xi$ blue or yellow with probability $1/2$, and we choose the right half-edge if $\xi$ is blue, or left half-edge if $\xi$ is yellow;
\item if $\xi$ is in $\{A, B, C, D\}$ and there is only one allowable mid-edges, we choose this allowable mid-edge.
\end{itemize}
We denote the endpoint of the new chosen half-edge or the new chosen mid-edge by $\gamma_{\delta}^{\prime}(k)$. The algorithm stops when $\gamma_{\delta}^{\prime}$ reaches $a_{\delta}$. This coloring algorithm generates paths with the same distribution as the radial exploration
precess from $b_{\delta}$ to $a_{\delta}$ because: for $k\geq 2$ and $k$ is not the first hitting time of $\partial D_{\delta}$,
 we color $\xi$ at the $k$-th step if and only if there are two allowable mid-edges at the $k$-th step.

See figure \ref{fig2} (respectively, figure \ref{fig3}) for a realization of $\gamma_{\delta}$ (respectively, $\gamma_{\delta}^{\prime}$). Note that neither $\gamma_{\delta}$ nor $\gamma_{\delta}^{\prime}$ is the interface separating yellow hexagons from blue hexagons.

From the coloring algorithm for $\gamma_{\delta}$, we can see $\gamma_{\delta}(k), k\geq 2$ has only one choice if and only if the hexagon $\xi$ centered at $\gamma_{\delta}(k-1)+3[\gamma_{\delta}(k-1)-\gamma_{\delta}(k-3/2)]$ has been colored or is in the complement of $D_{\delta}$. Therefore, each walk $\gamma_{\delta}$ has weight $(1/2)^{l(\gamma_{\delta})}$ where $l(\gamma_{\delta})$ is the number of colored hexagons in $D_{\delta}$ produced by the coloring algorithm for $\gamma_{\delta}$, which is also the number of hexagons in $D_{\delta}$ that share at least a half-edge with $\gamma_{\delta}$.

Similarly, the weight of each $\gamma_{\delta}^{\prime}$ is $1/4*(1/2)*(1/2)^{l^{\prime}(\gamma_{\delta}^{\prime})}$ where $l^{\prime}(\gamma_{\delta}^{\prime})$ is the number of colored hexagons in $D_{\delta}$ produced by the coloring algorithm for $\gamma_{\delta}^{\prime}$, here the factor $1/4$ comes from the first step (i.e., $\gamma_{\delta}^{\prime}(1)$) and the factor $1/2$ comes from the first time $\gamma_{\delta}^{\prime}$ hits $\partial D_{\delta}$ (since the corresponding $\xi$ for the first hitting of $\partial D_{\delta}$ is not in $D_{\delta}$ but we still have two choices).

For any simple polygonal path $\omega$ from $a_{\delta}$ to $b_{\delta}$, we claim that t $l(\omega)=l^{\prime}(\omega)+3$. Recall that $l(\omega)$ (respectively, $l^{\prime}(\omega)$) is the number of colored hexagons in $D_{\delta}$ produced by the coloring algorithm for $\gamma_{\delta}$ (respectively, $\gamma_{\delta}^{\prime}$) on the event $\gamma_{\delta}=\gamma_{\delta}^{\prime}=\omega$. The claim is true because:
\begin{itemize}
\item On the event $\gamma_{\delta}=\gamma_{\delta}^{\prime}=\omega$, for any hexagon in $D_{\delta}\setminus\{A,B,C,D\}$, either it is colored by both the coloring algorithm for $\gamma_{\delta}$ and the coloring algorithm for $\gamma_{\delta}^{\prime}$ or by neither. Actually, in $D_{\delta}\setminus\{A,B,C,D\}$, only those hexagons that share at least an edge with $\omega$ are colored.
    \item Number of hexagons in $\{A,B,C,D\}$ that share at least a half-edge with $\omega$ is either 3 or 4.
    \begin{itemize}
    \item If this number is 3, on the event $\gamma_{\delta}=\gamma_{\delta}^{\prime}=\omega$, the coloring algorithm for $\gamma_{\delta}$ colors three hexagons in $\{A,B,C,D\}$, in which case, the coloring algorithm for $\gamma_{\delta}^{\prime}$ colors none of $\{A,B,C,D\}$;
        \item If this number is 4, on the event $\gamma_{\delta}=\gamma_{\delta}^{\prime}=\omega$, the coloring algorithm for $\gamma_{\delta}$ colors all four hexagons in $\{A,B,C,D\}$, in which case, the coloring algorithm for $\gamma_{\delta}^{\prime}$ colors one hexagon of $\{A,B,C,D\}$.
    \end{itemize}
\end{itemize}

Note that in figure \ref{fig3}, at time t, the hexagon $\xi$ (which is labeled by B) centered at $\gamma_{\delta}^{\prime}(t)+[\gamma_{\delta}^{\prime}(t)-\gamma_{\delta}^{\prime}(t-1)]$ is uncolored because there is only one allowable mid-edge. Among the four hexagons (A,B,C,D) within distance $\delta$ from $b_{\delta}$, B, C and D are not colored. At time T, the first hitting time of $\gamma_{\delta}^{\prime}$ with $\partial D_{\delta}$, the hexagon $\xi$ centered at $\gamma_{\delta}^{\prime}(T)+[\gamma_{\delta}^{\prime}(T)-\gamma_{\delta}^{\prime}(T-1)]$ is not in $D_{\delta}$ but there are two allowable mid-edge. Therefore, we arrive at the following lemma:

\begin{lemma}\label{lem3.1}
Suppose $D_{\delta}$ is a simply connected domain in $\mathcal{H}_{\delta}$. For any simple polygonal path $\omega$ from $a_{\delta}$ to $b_{\delta}$, we have
$$P(\gamma_{\delta}=\omega)=P(\gamma_{\delta}^{\prime}=\omega)=(1/2)^{l(\omega)},$$
where $l(\omega)$ is the number of hexagons in $D_{\delta}$ sharing at least a half-edge with $\omega$.
\end{lemma}

\begin{proof}[Proof of Lemma \ref{lem1}]
Lemma \ref{lem1} follows directly from Lemma \ref{lem3.1}.
\end{proof}

\section{Convergence of radial exploration process}\label{sec4}
In this section, we will prove Theorem \ref{thm1} using the strategy outlined by Werner \cite{Wer07} and the techniques developed by Camia and Newman (\cite{CN06} and \cite{CN07}). The idea of the proof is to find a family of stopping times for the radial exploration process and the radial SLE$_6$, and then we will show the discrete process converges to the corresponding continuous one in each time interval. We will assume $D=\mathbb{D}$ and $a=1$, $b=0$ since the proof for general Jordan domains is similar.

\begin{figure}
\centering
\includegraphics[height=2.5in,width=3in]{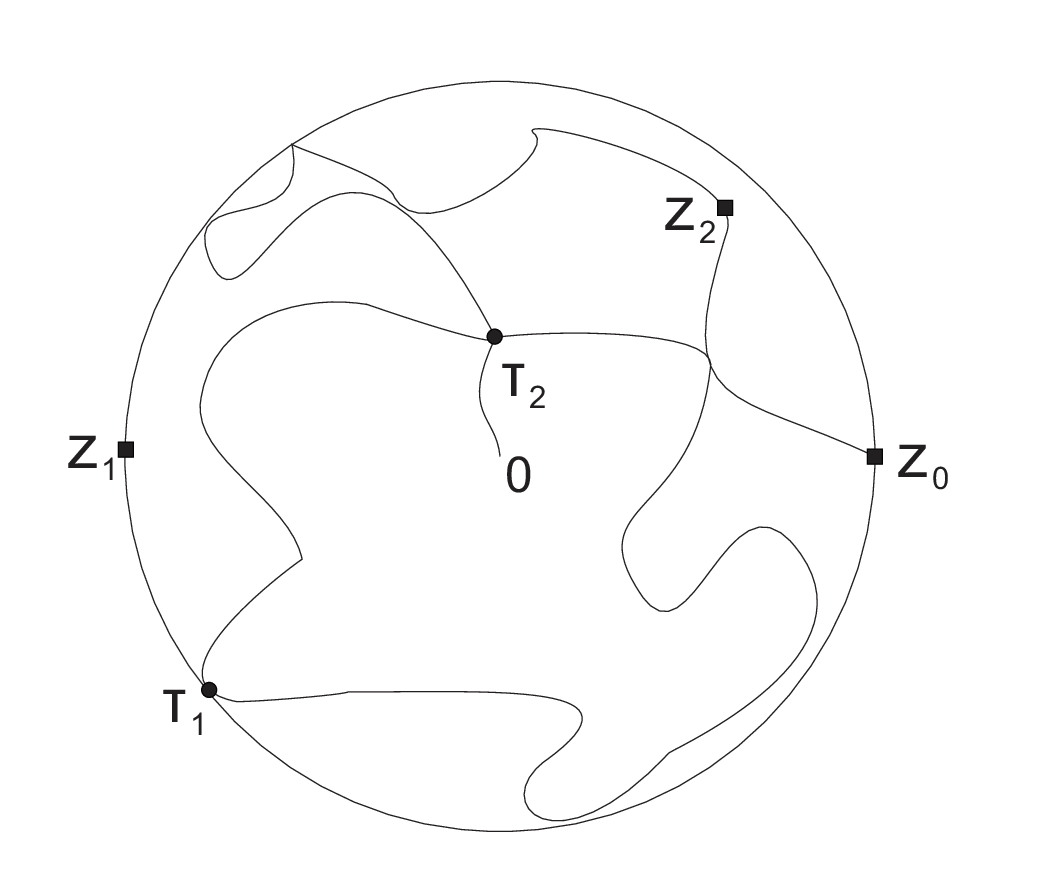}
\caption{The continuous construction }\label{fig4}
\end{figure}

\subsection{The continuous construction}
Let $\gamma(t), 0\leq t<\infty$ be the trace of the radial SLE$_6$ in $\mathbb{D}$ from $a=1$ to $b=0$. For any domain $D\subseteq \mathbb{D}$, let $d_x(D)$ and $d_y(D)$ be respectively the maximal $x-$ and $y-$ distance between pairs of points in $D$, i.e., $d_x(D):=\sup\{|\text{Re}(w-\tilde{w})|:w,\tilde{w}\in D\}$. In the first step, let $z_0=1$ and $D_1=\mathbb{D}$. If $d_x(D_1)\geq d_y(D_1)$ then we choose any $z_1\in\partial D_1$ satisfying $|\text{Re}(z_1-z_0)|\geq d_x(D_1)/2 $, otherwise we choose any $z_1\in\partial D_1$ satisfying $|\text{Im}(z_1-z_0)|\geq d_y(D_1)/2 $; let $\tau_1=\inf\{t\geq0: \text{ there is no path from } 0 \text{ to } z_1 \text{ in } D_1\setminus\gamma[0,t]\}$; let $D_2$ be the connected component of $D_1\setminus\gamma[0,\tau_1]$ that contains $0$. In the $k$-th ($k\geq 2$) step, if $d_x(D_k)\geq d_y(D_k)$ then we choose any $z_k\in\partial D_k$ satisfying $|\text{Re}(z_k-\gamma(\tau_{k-1}))|\geq d_x(D_k)/2 $, otherwise we choose any $z_k\in\partial D_k$ satisfying $|\text{Im}(z_k-\gamma(\tau_{k-1}))|\geq d_y(D_k)/2 $; let $\tau_k=\inf\{t\geq\tau_{k-1}:\text{ there is no path from } 0 \text{ to } z_k \text{ in } D_k\setminus\gamma[\tau_{k-1},t] \}$; let $D_{k+1}$ be the connected component of $D_k\setminus\gamma[\tau_{k-1},\tau_k]$ that contains $0$. See figure \ref{fig4} for an illustration of the first two steps of the continuous construction. We are not able to prove for $k\geq 2$, $d_x(D_k)\neq d_y(D_k)$ a.s., and this is largely responsible for the lengthy proof of Lemma \ref{lem4.2}.

\subsection{The discrete construction}
The discrete construction is based on the continuous construction, so we will use notations defined in the continuous construction. Let $D_1^{\delta}$ be the largest connected component of hexagons of $\mathbb{D}\cap \mathcal{H}_{\delta}$ where $\mathcal{H}_{\delta}$ is the hexagonal lattice with mesh $\delta$. Let $z_0^{\delta}$ be a closest mid-edge to $z_0=1$ in the set of mid-edges outside of $D_1^{\delta}$ but within $\delta/2$ distance from the topological boundary of $D_1^{\delta}$, $\partial D_1^{\delta}$, and let $b^{\delta}$ be a closest mid-edge to $b=0$ in $D_1^{\delta}$. Let $\gamma_{\delta}$ be the radial exploration process from $z_0^{\delta}$ to $b^{\delta}$ in $D_1^{\delta}$ (see the definition in the introduction). In the first step, if $d_x(D_1^{\delta})\geq d_y(D_1^{\delta})$ and $d_x(D_1)\neq d_y(D_1)$ then we choose $z_1^{\delta}$ to be any mid-edge $z$ in $\{z\in\partial D_1^{\delta}: |\text{Re}(z-z_0^{\delta})|\geq d_x(D_1^{\delta})/2 \}$ such that $z$ minimizes $|z-z_1|$; if $d_x(D_1^{\delta})< d_y(D_1^{\delta})$ and $d_x(D_1)\neq d_y(D_1)$ then we choose $z_1^{\delta}$ to be any mid-edge $z$ in $\{z\in\partial D_1^{\delta}: |\text{Im}(z-z_0^{\delta})|\geq d_y(D_1^{\delta})/2 \}$ such that $z$ minimizes $|z-z_1|$; otherwise (i.e., $d_x(D_1)= d_y(D_1)$) we choose $z_1^{\delta}$ to be any mid-edge $z$ in $\{z\in\partial D_1^{\delta}: |\text{Re}(z-z_0^{\delta})|\geq d_x(D_1^{\delta})/2 \}$ such that $z$ minimizes $|z-z_1|$. Let $\tau_1^{\delta}=\inf\{t\geq0:\text{there is no polygonal path from }b^{\delta} \text{ to } z_1^{\delta} \text{ in } \overline{D_1^{\delta}\setminus \Gamma(\gamma_{\delta}[0,t])} \}$ where $\Gamma(\gamma_{\delta}[0,t])$ is the set of hexagons in $D_1^{\delta}$ sharing at least an edge with $\gamma_{\delta}[0,t]$, and note that here we view $D_1^{\delta}$ and $\Gamma(\gamma_{\delta}[0,t])$ as subsets of $\mathbb{C}$ and the overline means the closure. Let $D_2^{\delta}$ be the connected component of $\overline{D_1^{\delta}\setminus \Gamma(\gamma_{\delta}[0,\tau_1^{\delta}])}$ that contains $b^{\delta}$. Let $\sigma_1^{\delta}:=\inf\{t\geq0: \gamma_{\delta}(t)\in D_2^{\delta}\}$. Note that  $\gamma_{\delta}(\sigma_1^{\delta})$ and  $\gamma_{\delta}(\tau_1^{\delta})$ are on the boundary of the same hexagon. See figure \ref{fig5} for an illustration of the first step of the discrete construction.
In the $k$-th ($k\geq 2$) step, if $d_x(D_k^{\delta})\geq d_y(D_k^{\delta})$ and $d_x(D_k)\neq d_y(D_k)$ then we choose $z_k^{\delta}$ to be any mid-edge $z$ in $\{z\in\partial D_k^{\delta}: |\text{Re}(z-\gamma_{\delta}(\sigma_{k-1}^{\delta}))|\geq d_x(D_k^{\delta})/2 \}$ such that $z$ minimizes $|z-z_k|$; if $d_x(D_k^{\delta})< d_y(D_k^{\delta})$ and $d_x(D_k)\neq d_y(D_k)$ then we choose $z_k^{\delta}$ to be any mid-edge $z$ in $\{z\in\partial D_k^{\delta}: |\text{Im}(z-\gamma_{\delta}(\sigma_{k-1}^{\delta}))|\geq d_y(D_k^{\delta})/2 \}$ such that $z$ minimizes $|z-z_k|$; otherwise (i.e., $d_x(D_k)=d_y(D_k)$) we choose $z_k^{\delta}$ to be any mid-edge $z$ in $\{z\in\partial D_k^{\delta}: |\text{Re}(z-\gamma_{\delta}(\sigma_{k-1}^{\delta}))|\geq d_x(D_k^{\delta})/2 \}$ such that $z$ minimizes $|z-z_k|$. Let $\tau_k^{\delta}=\inf\{t\geq\sigma_{k-1}^{\delta}:\text{there is no polygonal path from }b^{\delta} \text{ to } z_k^{\delta} \text{ in } \overline{D_k^{\delta}\setminus \Gamma(\gamma_{\delta}[\sigma_{k-1}^{\delta},t])} \}$. Let $D_{k+1}^{\delta}$ be the connected component of $\overline{D_k^{\delta}\setminus \Gamma(\gamma_{\delta}[\sigma_{k-1}^{\delta},\tau_k^{\delta}])}$ that contains $b^{\delta}$. Let $\sigma_k^{\delta}:=\inf\{t\geq0: \gamma_{\delta}(t)\in D_{k+1}^{\delta}\}$.

\begin{figure}
\centering
\includegraphics[height=2.5in,width=3.2in]{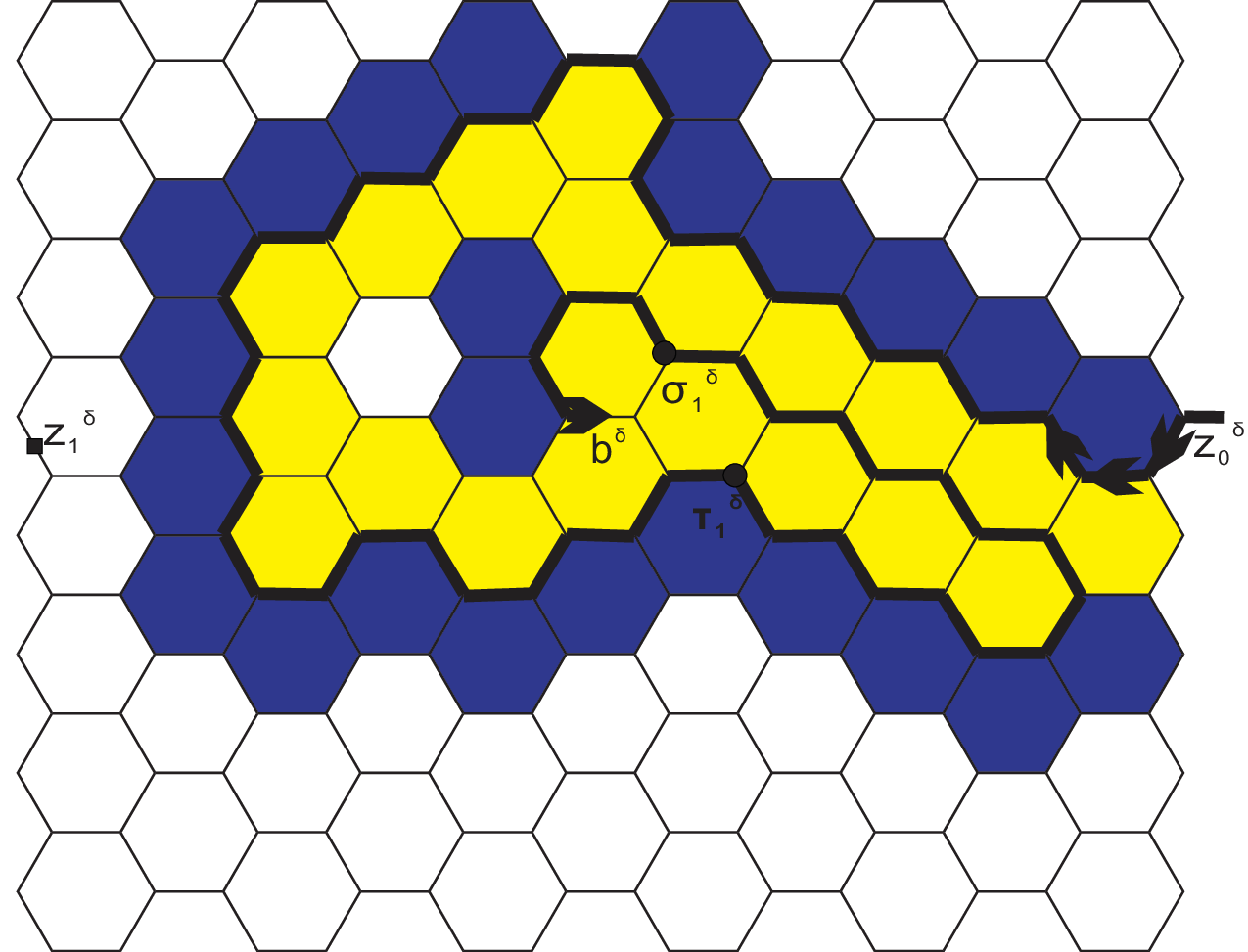}
\caption{The discrete construction }\label{fig5}
\end{figure}

\subsection{Proof of Theorem \ref{thm1}}
The following lemma says that there is no difference between $\tau_1^{\delta}$ and $\sigma_1^{\delta}$ when $\delta$ approaches $0$.
\begin{lemma}\label{lem4.1}
$(\gamma_{\delta}[0,\tau_1^{\delta}],\partial D_2^{\delta})$ converges jointly in distribution to $(\gamma[0,\tau_1],\partial D_2)$. And\\ $(\gamma_{\delta}[0,\sigma_1^{\delta}],\partial D_2^{\delta})$ converges jointly in distribution to $(\gamma[0,\tau_1],\partial D_2)$
\end{lemma}
\begin{proof}
Note that $\gamma_{\delta}[0,\tau_1^{\delta}]$ has the same distribution as the chordal exploration process (say $\tilde{\gamma}_{\delta}$) from $z_0^{\delta}$ to $z_1^{\delta}$ (one needs to shift $z_1^{\delta}$ by distance $\delta/2$, but we still use the same letter and this should not cause any confusion) in $D_1^{\delta}$ up to a similar defined stopping time $\tilde{\tau}_1^{\delta}$. Moreover, the
locality property of SLE$_6$ (see Proposition 4.1 of \cite{Wer07}) implies $\gamma[0,\tau_1]$ has the same distribution as the chordal SLE$_6$ (say $\tilde{\gamma}$) from $z_0$ to $z_1$ in $D_1$ up to a similar defined stopping time $\tilde{\tau}_1$. From Theorem 5 of \cite{CN07}, we know $\tilde{\gamma}_{\delta}$ converges weakly to $\tilde{\gamma}$. So Theorem 6.7 of \cite{Bil99} implies we can find coupled versions of $\tilde{\gamma}_{\delta}$ and $\tilde{\gamma}$ on the same probability space such that $d(\tilde{\gamma}_{\delta},\tilde{\gamma})\rightarrow 0$ a.s. as $\delta\downarrow 0$. Under this coupling, we claim
$$\lim_{\delta\downarrow 0}\tilde{\gamma}_{\delta}(\tilde{\tau}_1^{\delta})=\tilde{\gamma}(\tilde{\tau}_1)\text{ a.s.}$$
This is actually Lemma 3.1 of \cite{Wer07}, and we will give a slightly different argument. It is obvious that any subsequential limit of $\tilde{\gamma}_{\delta}(\tilde{\tau}_1^{\delta})$ is in $\tilde{\gamma}[0,\infty]$, so we can define a linear ordering on $\tilde{\gamma}[0,\infty]$ such that $\tilde{\gamma}(t_1)\leq\tilde{\gamma}(t_2)$ if $t_1\leq t_2$. Under this ordering clearly we have  $\liminf_{\delta\downarrow 0}\tilde{\gamma}_{\delta}(\tilde{\tau}_1^{\delta})\in\tilde{\gamma}[\tilde{\tau}_1,\infty]$. On the other hand, suppose $\limsup_{\delta\downarrow 0}\tilde{\gamma}_{\delta}(\tilde{\tau}_1^{\delta})\in\tilde{\gamma}(\tilde{\tau}_1,\infty]$, then along some subsequence of $\delta$ either the 6-arm (not all of the same color) event occurs in $\mathbb{D}$ or the 3-arm (not all of the same color) event occurs on $\partial \mathbb{D}$, which contradicts Lemma 6.1 of \cite{CN06}, and thus  $\limsup_{\delta\downarrow 0}\tilde{\gamma}_{\delta}(\tilde{\tau}_1^{\delta})\in\tilde{\gamma}[0,\tilde{\tau}_1]$. Therefore the claim follows. It is clear that $\tilde{\gamma}(\tilde{\tau}_1)\notin\tilde{\gamma}[0,\infty)\setminus\{\tilde{\gamma}(\tilde{\tau}_1)\}$ a.s. since otherwise $\tilde{\gamma}$ would hit the same boundary point twice or have a triple point. Therefore, $\tilde{\gamma}_{\delta}[0,\tilde{\tau}_1^{\delta}]$ converges a.s to $\tilde{\gamma}[0,\tilde{\tau}_1]$. In particular, this implies $\gamma_{\delta}[0,\tau_1^{\delta}]$ converges weakly to $\gamma[0,\tau_1]$ in the metric \eqref{eq2.1}. Let $\tilde{D}_2^{\delta}$ be the unique domain in $\overline{D_1^{\delta}\setminus\Gamma(\tilde{\gamma}_{\delta}[0,\tilde{\tau}_1^{\delta}])}$ that contains $b^{\delta}$. Then Lemma 5.2 of \cite{CN06} implies $\partial \tilde{D}_2^{\delta}$ converges weakly to $\partial D_2$. Note that $\partial D_2^{\delta}$ has the same distribution as $\partial \tilde{D}_2^{\delta}$, and thus $\partial D_2^{\delta}$ converges weakly to $\partial D_2$. So the first part the lemma follows. For the second part of the lemma, note that $\gamma_{\delta}(\sigma_1^{\delta})$ and  $\gamma_{\delta}(\tau_1^{\delta})$ are on the boundary of the same hexagon. Moreover we have $\lim_{\delta\downarrow 0}\gamma_{\delta}[\tau_1^{\delta},\sigma_1^{\delta}]=\gamma(\tau_1)$ since otherwise the 6-arm event would occur which contradicts Lemma 6.1 of \cite{CN06}.
\end{proof}

Next, we extend Lemma \ref{lem4.1} to all $k\in\mathbb{N}$.
\begin{theorem}\label{thm4.1}
For any $k\in\mathbb{N}$, $\gamma_{\delta}[0,\tau_k^{\delta}]$ converges weakly to $\gamma[0,\tau_k]$ and $\gamma_{\delta}[0,\sigma_k^{\delta}]$ converges weakly to $\gamma[0,\tau_k]$. Moreover, for any $\epsilon>0$, let $K_{\delta}^{\epsilon}:=\inf\{k\geq 1: \text{diam}(D_k^{\delta}):=\sup\{|x-y|:x,y\in D_k^{\delta}\}<\epsilon\}$, then
$$\lim_{C\rightarrow\infty}\limsup_{\delta\downarrow 0}P(K_{\delta}^{\epsilon}>C)=0.$$
\end{theorem}
\begin{remark}
This theorem implies that $\lim_{k\rightarrow\infty}\text{diam}(D_k)=0$.
\end{remark}
\begin{proof}
We prove the first part of the theorem by induction in $k$.

$k=1$. This is Lemma \ref{lem4.1}.

$k=2$. It follows from \cite{AB99} that $(\gamma_{\delta}[0,\tau_1^{\delta}],\partial D_2^{\delta},\gamma_{\delta}[\sigma_1^{\delta},\tau_2^{\delta}])$ converges jointly in distribution along some subsequence to some limit $(\tilde{\gamma}_1,\partial \tilde{D_2},\tilde{\gamma}_2)$. Lemma \ref{lem4.1} implies $\tilde{\gamma}_1$ is distributed like $\gamma[0,\tau_1]$ and $\partial \tilde{D_2}$ is distributed like $\partial D_2$. Corollary 5.1 and Lemma 5.3 of \cite{CN06}, and similar argument as Lemma \ref{lem4.1} imply $\tilde{\gamma}_2$ is distributed like $\gamma[\tau_1,\tau_2]$. Therefore, we have $(\gamma_{\delta}[0,\tau_1^{\delta}],\partial D_2^{\delta},\gamma_{\delta}[\sigma_1^{\delta},\tau_2^{\delta}])$ converges jointly in distribution to $(\gamma[0,\tau_1],\partial D_2,\gamma[\tau_1,\tau_2])$. Applying Theorem 6.7 of \cite{Bil99}, we see that $\gamma^{\delta}[0,\tau_2^{\delta}]$ converges in distribution (or weakly) to $\gamma[0,\tau_2]$.

$k\geq 3$. All steps for $k\geq 3$ are analogous to the case $k=2$.
For the proof of the second part of the theorem, we need the following lemma.

\begin{lemma}\label{lem4.2}
For any $k\in\mathbb{N}$ and any $\epsilon>0$, if $\max_{j=k,k+1,k+2}\max\{d_x(D_j^{\delta}),d_y(D_j^{\delta})\}\geq\epsilon$ then we have
\begin{equation}\label{eq4.1}
\max\{d_x(D_{k+3}^{\delta}),d_y(D_{k+3}^{\delta})\}\leq \frac{23}{24}\max\{d_x(D_{k}^{\delta}),d_y(D_{k}^{\delta})\}
\end{equation}
with probability at least $p_0$ independent of $\delta$, i.e., \eqref{eq4.1} is true for any $\delta\leq\delta(k)$ for each fixed $k$ where $\delta(k)>0$.
\end{lemma}
\begin{proof}
The basic idea of the proof is from the proof of Lemma 6.4 of \cite{CN06}.
\begin{itemize}
\item
Case 1: $d_x(D_k^{\delta})\geq d_y(D_k^{\delta})$ and $d_x(D_k)\neq d_y(D_k)$. We know $z_k^{\delta}$ satisfies $|\text{Re}(z_k^{\delta}-\gamma_{\delta}(\sigma_{k-1}^{\delta}))|\geq d_x(D_k^{\delta})/2 $. Consider the rectangle $R$ (see figure \ref{fig6}) whose vertical sides are parallel to the $y$-axis, have length $d_x(D_k^{\delta})$, and are each placed between the $x$-coordinates of $z_k^{\delta}$ and $\gamma_{\delta}(\sigma_{k-1}^{\delta})$ such that the horizontal sides of $R$ have length $d_x(D_k^{\delta})/3$; the bottom and top sides of $R$ are placed in such a way that they are equal $y$-distance from the points of $\partial D_k^{\delta}$ with minimal or maximal $y$-coordinate, respectively. Denote $l$ by the line passes through the midpoints of the bottom and top sides of $R$. Then $b^{\delta}$ is located either to the right of  $l$ or to the left of $l$. Without loss of generality, we assume $b^{\delta}$ is located to the right of $l$. Divide the left half of $R$ into 2 congruent rectangle with width $d_x(D_k^{\delta})/12$ and height $d_x(D_k^{\delta})$, and label them by $B$ and $C$. Note that $\gamma_{\delta}[\sigma_{k-1}^{\delta},\tau_k^{\delta}]$ has the distribution of the chordal exploration process from $\gamma_{\delta}(\sigma_{k-1}^{\delta})$ to $z_k^{\delta}$ in $D_k^{\delta}$ up to the stopping time that $b^{\delta}$ is disconnect from $z_k^{\delta}$. It follows from the Russo-Seymour-Welsh lemma \cite{Rus78,SW78} that the probability to have vertical crossing of $C$ of different colors is bounded away from zero by a positive constant $p_0$ that does not depend on $\delta$ (for $\delta$ small enough). Recall that $D_{k+1}^{\delta}$ is a connected component of $\overline{D_k^{\delta}\setminus \Gamma(\gamma_{\delta}[\sigma_{k-1}^{\delta},\tau_k^{\delta}])}$.  After $\gamma_{\delta}$ completes the interface in $C$ (say at time $t_0$), there is no polygonal path from $z_k^{\delta}$ to $b^{\delta}$ in $\overline{D_k^{\delta}\setminus \Gamma(\gamma_{\delta}[\sigma_{k-1}^{\delta},t_0)}$, and thus $\tau_k^{\delta}$ happens before the completion of the interface in $C$, which implies $B$ is not contained in $D_{k+1}^{\delta}$, so $d_x(D_{k+1}^{\delta})\leq 11d_x(D_k^{\delta})/12 $ with probability at least $p_0$ independent of $\delta$.

\begin{figure}
\centering
\includegraphics[height=2.8in,width=3.0in]{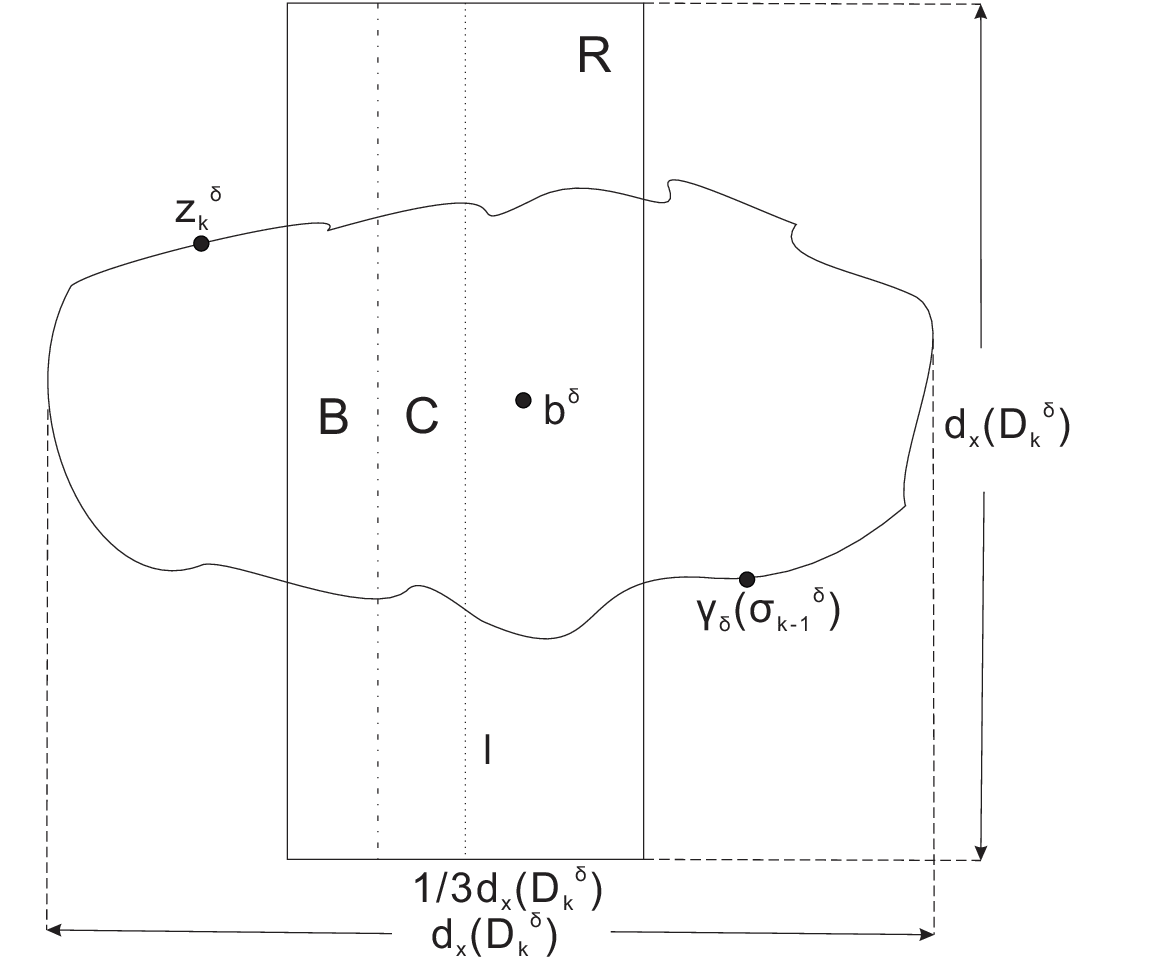}
\caption{The rectangle $R$ }\label{fig6}
\end{figure}
\item
Case 2: $d_x(D_k^{\delta})< d_y(D_k^{\delta})$ and $d_x(D_k)\neq d_y(D_k)$. Similar argument as case 1 implies $d_y(D_{k+1}^{\delta})\leq 11d_y(D_k^{\delta})/12$ with probability at least $p_0$ independent of $\delta$.
\item
Case 3: $d_x(D_k^{\delta})\geq d_y(D_k^{\delta})$ and $d_x(D_k)= d_y(D_k)$. The conclusion of case 1 is also valid here.
\item
Case 4: $d_x(D_k^{\delta})< d_y(D_k^{\delta})$ and $d_x(D_k)= d_y(D_k)$. For fixed $k$, by the first part of the theorem and Theorem 6.7 of \cite{Bil99}, there are coupled versions of $\gamma_{\delta}[0,\tau_k^{\delta}]$ and $\gamma[0,\tau_k]$ on the same probability space such that $d(\gamma_{\delta}[0,\tau_k^{\delta}], \gamma[0,\tau_k])\rightarrow 0$ a.s. as $\delta\downarrow 0$. Under this coupling, we have $|d_x(D_k^{\delta})-d_y(D_k^{\delta})|<\epsilon/2$ when $\delta$ is small enough. Since we assumed $\max\{d_x(D_k^{\delta}),d_y(D_k^{\delta})\}\geq \epsilon$, we get $d_x(D_k^{\delta})>\epsilon/2$. Similar rectangle as in case 1 with horizontal length $d_x(D_k^{\delta})/3$ and vertical length $d_y(D_k^{\delta})$ (note that $d_y(D_k^{\delta})<d_x(D_k^{\delta})+\epsilon/2<2d_x(D_k^{\delta})$) implies that $d_x(D_{k+1}^{\delta})\leq 11d_x(D_k^{\delta})/12$ with probability at least $p_0$ independent of $\delta$ (one may need to change $p_0$ to a new positive number here).
\end{itemize}

If both step $k$ and step $k+1$ are in cases 1, 2 and 3 then the lemma follows since cases 1, 2 and 3 reduce the maximum of $x$- and $y-$ distance by at least a factor of $1/12$.

If step $k$ is in case 4, step $k+1$ is in case 1, then we have
\begin{eqnarray}\label{eq4.2}
&&\max\{d_x(D_{k+2}^{\delta}),d_y(D_{k+2}^{\delta})\}\leq\max\{11d_x(D_{k+1}^{\delta})/12,d_y(D_{k+1}^{\delta})\}\nonumber\\
&&\leq\max\{11d_x(D_{k+1}^{\delta})/12,d_x(D_{k+1}^{\delta})\}\leq\max\{11d_x(D_{k}^{\delta})/12,11d_x(D_{k}^{\delta})/12\}\nonumber\\
&&\leq 11/12\max\{d_x(D_{k}^{\delta}),d_y(D_{k}^{\delta})\},
\end{eqnarray}
so the lemma follows.

If step $k$ is in case 4, step $k+1$ is in case 3, then the lemma follows by a similar argument as \eqref{eq4.2}. If step $k$ is in case 4, step $k+1$ is in case 2, then the proof of the lemma is trivial. If step $k$ is in case 2, step $k+1$ is in case 4, then the proof of the lemma is also trivial.

If step $k$ and step $k+1$ are in case 4: if $d_y(D_{k+1}^{\delta})\leq 23d_y(D_{k}^{\delta})/24$ then we are done, otherwise we have
$$d_y(D_{k+1}^{\delta})-d_x(D_{k+1}^{\delta})>23d_y(D_{k}^{\delta})/24-11d_x(D_{k}^{\delta})/12>d_y(D_{k}^{\delta})/24\geq\epsilon/24,$$
which contradicts the fact $d_y(D_{k+1})=d_x(D_{k+1})$ when $\delta$ is small.

So the only two bad situations that we can not achieve \eqref{eq4.1} in two steps are step $k$ and step $k+1$ are in case 1 and case 4 respectively, and in case 3 and case 4 respectively. But if we look into step $k+2$ then the lemma follows since we already proved any two successive steps starts with case 4 will reduce the maximum of $x$- and $y$- distances by a factor of $11/12$.
\end{proof}
Let $\tilde{K}_{\delta}^{\epsilon}:=\inf\{k\geq 1: \max\{d_x(D_{k}^{\delta}),d_y(D_{k}^{\delta})\}<\epsilon/\sqrt{2}\}$. Then $K_{\delta}^{\epsilon}\leq \tilde{K}_{\delta}^{\epsilon}$. Note that $\tilde{K}_{\delta}^{\epsilon}$ only depends on $\{k:\max\{d_x(D_{k}^{\delta}),d_y(D_{k}^{\delta})\}\geq\epsilon/\sqrt{2}\}$. Let $H(\epsilon)$ be the smallest integer $h\geq 1$ such that $2(23/24)^{h+1}<\epsilon/\sqrt{2}$ where the $2$ on the left hand side is the diameter of $\mathbb{D}$, i.e., $2(23/24)^{H(\epsilon)+1}<\epsilon/\sqrt{2}$ and $2(23/24)^{H(\epsilon)}\geq\epsilon/\sqrt{2}$. we call the first 3 steps of our discrete construction the 1st trial, and the steps $3k-2,3k-1,3k$ of our discrete construction the $k$-th trial ($k\geq 1$). We say the $k$-th trial is \textit{successful} if the event $E_k$ defined by \eqref{eq4.1} occurs. Lemma \ref{lem4.2} implies $P(E_1^c)\leq 1-p_0$ and $P(E_1^c|E_2^c)\leq 1-p_0$, so $P(E_1^cE_2^2)\leq (1-p_0)^2$. A simple induction argument gives $P(E_{i_1}^cE_{i_2}^c\cdots E_{i_j}^c)\leq (1-p_0)^j$ for any $j\in\mathbb{N}$ and $1\leq i_1<i_2<\cdots<i_j\leq 3n$ where $n$ is a fixed integer. Therefore,
\begin{eqnarray*}
P(\tilde{K}_{\delta}^{\epsilon}\geq 3n) &\leq& P\left(\text{number of successes in the } n \text{ trials }\leq H(\epsilon)\right)\\
&=&\sum_{k=0}^{H(\epsilon)}P\left(\text{number of successes in the } n \text{ trials }=k\right)\\
&\leq&\sum_{k=0}^{H(\epsilon)}\binom{n}{k}(1-p_0)^{n-k}\leq (H(\epsilon)+1)n^{H(\epsilon)}(1-p_0)^{n-H(\epsilon)}
\end{eqnarray*}
where the last term approaches $0$ as $n\rightarrow\infty$ since $H(\epsilon)$ is fixed when $\epsilon$ is fixed. Therefore
$$\lim_{C\rightarrow\infty}\limsup_{\delta\downarrow 0}P(\tilde{K}_{\delta}^{\epsilon}>C)=0.$$
This finishes the proof of the theorem since $K_{\delta}^{\epsilon}\leq \tilde{K}_{\delta}^{\epsilon}$.
\end{proof}
\begin{proof}[Proof of the Theorem \ref{thm1}]
Theorem \ref{thm1} is a immediate consequence of Theorem \ref{thm4.1}.
\end{proof}

\section{Convergence of full-plane exploration process}\label{sec5}
In this section, we will prove Theorem \ref{thm2}. Recall the definition of full-plane exploration process in $\mathcal{H}_{\delta}$ from $0_{\delta}$ to $\infty$ in the introduction. We will need Corollary \ref{cor}, so we prove it here
\begin{proof}[Proof of Corollary \ref{cor}]
Since radial SLE$_6$ in any Jordan domain is defined by the conformal image of the radial SLE$_6$ in $\mathbb{D}$, it suffices to prove Corollary \ref{cor} for $D=\mathbb{D}$. Let $\gamma$ be the radial SLE$_6$ trace from $1$ to $0$ in $\mathbb{D}$. Let $T:=\sup\{t\geq0:\gamma(t)\in\partial \mathbb{D}\}$. Our goal is to prove the distribution on $\partial \mathbb{D}$ induced by $\gamma(T)$ is the uniform distribution. Let $\mathbb{D}_{\delta}$ be the largest connected component of hexagons of $\mathbb{D}\cap\mathcal{H}_{\delta}$. Let $1_{\delta}$ be a closest mid-edge to $1$ in the set of mid-edges outside of $\mathbb{D}_{\delta}$ but within $\delta/2$ distance from $\partial \mathbb{D}_{\delta}$, and let $0_{\delta}$ be a closest mid-edge to $0$ in $\mathbb{D}_{\delta}$. Let $\gamma_{\delta}$ be a radial exploration process from $1_{\delta}$ to $0_{\delta}$ in $\mathbb{D}_{\delta}$, and $\gamma_{\delta}^{\prime}$ be a radial exploration process from $0_{\delta}$ to $1_{\delta}$ in $\mathbb{D}_{\delta}$. Let $T_{\delta}:=\sup\{t\geq 0:\gamma_{\delta}(t)\in \partial \mathbb{D}_{\delta}\}$, and $T_{\delta}^{\prime}:=\inf\{t\geq 0:\gamma_{\delta}^{\prime}(t)\in \partial \mathbb{D}_{\delta}\}$. Then Lemma \ref{lem3.1} implies $\gamma_{\delta}$ and the time-reversal of $\gamma_{\delta}^{\prime}$ have the same distribution, and thus $\gamma_{\delta}(T_{\delta})$ and $\gamma_{\delta}^{\prime}(T_{\delta}^{\prime})$ have the same distribution. By Theorem \ref{thm1} of this paper and Theorem 6.7 of \cite{Bil99}, we can find coupled versions of $\gamma_{\delta}$ and $\gamma$ such that $d(\gamma_{\delta},\gamma)\rightarrow 0$ a.s. as $\delta\downarrow 0$. Under this coupling, we claim:
\begin{equation*}
\lim_{\delta\downarrow 0}\gamma_{\delta}(T_{\delta})=\gamma(T)\text{ a.s.}
\end{equation*}
The proof of the claim is similar to the proof of Lemma \ref{lem4.1}: it is clear that $\limsup_{\delta\downarrow 0}\gamma_{\delta}(T_{\delta})\in \gamma[0,T]$ where the $\limsup$ is defined by the linear ordering $\gamma(t_1)\leq\gamma(t_2)$ for any $t_1\leq t_2$; suppose $\liminf_{\delta\downarrow 0}\gamma_{\delta}(T_{\delta})\in \gamma[0,T)$, then along some subsequence of $\delta$ the 3-arm (not all of the same color) event occurs on $\partial\mathbb{D}$, which contradicts Lemma 6.1 of \cite{CN06}. Since $\gamma_{\delta}(T_{\delta})$ and $\gamma_{\delta}^{\prime}(T_{\delta}^{\prime})$ have the same distribution, we conclude that $\gamma_{\delta}^{\prime}(T_{\delta}^{\prime})$ converges in distribution to $\gamma(T)$ as $\delta\downarrow 0$. Note that the distribution of $\gamma_{\delta}^{\prime}(T_{\delta}^{\prime})$ does not change if we only change the endpoint of $\gamma_{\delta}^{\prime}$ to any point on the unit circle, which implies the distribution on $\partial\mathbb{D}$ induced by $\gamma(T)$ is the same as the distribution induced by $e^{i\theta}\gamma(T)$ for any $0\leq \theta<2\pi$. Therefore, the distribution induced by $\gamma(T)$ is uniform, and the corollary follows.
\end{proof}

Next, we generalize Theorem \ref{thm1} to unbounded Jordan domains.
\begin{proposition}\label{prop5.1}
Let $D$ be an unbounded Jordan domain in $\mathbb{C}$ that contains $\infty$ as an interior point (view as a subset of $\hat{\mathbb{C}}$) and  $a\in\partial D$. Let $\gamma_{\delta}$ be the radial exploration process in $D_{\delta}$ from $a_{\delta}$ to $\infty$ where $D_{\delta}$ and $a_{\delta}$ are defined as the bounded case. Let $\gamma$ be the radial SLE$_6$ in $D$ from $a$ to $\infty$. Then $\gamma_{\delta}$ converges weakly to $\gamma$ in the metric defined by \eqref{eq2.2}.
\end{proposition}
\begin{proof}
First of all, Cardy's formula (see \cite{Car92},\cite{Smi01} and also \cite{Bef07} for a easy proof) is valid for unbounded Jordan domains. So Theorem 5 of \cite{CN07} is also true for unbounded Jordan domains, i.e., the chordal exploration process in an unbounded Jordan domain converges to chordal SLE$_6$ in the same domain. The rest proof is the same as the proof of Theorem \ref{thm1}.
\end{proof}
\begin{remark}
A similar proof as the proof of Lemma 5.3 of \cite{CN06} gives:
Let $(D,a)$ be a random unbounded Jordan domain, with $a\in\partial D$. Let $\{(D_k,a_k)\}_{k\in\mathbb{N}}$, be a sequence of random Jordan domains with $a_k\in \partial D_k$ such that, as $k\rightarrow\infty$, $(\partial D_k, a_k)$ converges in distribution to $(\partial D,a)$ with respect the metric \eqref{eq2.1} on continuous curve, and the Euclidean metric on $a$. Let $\gamma_{\delta}^k$ be the radial exploration process in $(D_k)_{\delta}$ from $(a_k)_{\delta}$ to $\infty$. For any sequence $\{\delta_k\}_{k\in\mathbb{N}}$ with $\delta_k\downarrow 0$ as $k\rightarrow\infty$, then $\gamma_{\delta_k}^k$ converges weakly to the radial SLE$_6$ in $D$ from $a$ to $\infty$ with respect to metric \eqref{eq2.2}.
\end{remark}

We will need some properties about the full-plane SLE$_{\kappa}$.
\begin{lemma}\label{lem5.1}
Let $\gamma$ be the trace of the full-plane SLE$_{\kappa}$ in $\mathbb{C}$ from $0$ to $\infty$. For any fixed $s\in \mathbb{R}$, conditioned on the $\gamma[-\infty,s]$ , $\gamma[s,\infty]$ has the distribution of a radial SLE$_{\kappa}$ trace started at $\gamma(s)$ and growing in the connected component of $\hat{\mathbb{C}}\setminus\gamma[-\infty,s]$ that contains $\infty$. Moreover, suppose $\gamma^s$ is the radial SLE$_{\kappa}$ in $\mathbb{C}\setminus e^s\mathbb{D}$ started uniformly on $e^s\partial\mathbb{D}$. Then as $s\rightarrow -\infty$, $\gamma^s$ converges weakly to $\gamma$ in the metric \eqref{eq2.2}.
\end{lemma}
\begin{proof}
Let $g_t$ be the Loewner maps, i.e., $g_t$ satisfying \eqref{eq2.3} and the initial condition following it. We follow the idea in section 2.4 of \cite{Law04}. For any $s\leq t$, we define $h_{s,t}(z):=g_t(g_s^{-1}(z))$ for any $z$ in the unbounded component of $\mathbb{C}\setminus\{\overline{\mathbb{D}}\cup g_s(\gamma[s,t])\}$. Then we have
$$\partial_t h_{s,t}(z)=-h_{s,t}(z)\frac{h_{s,t}(z)+e^{-iU_t}}{h_{s,t}(z)-e^{-iU_t}},~~~h_{s,s}(z)=z.$$
So for any fixed $s$, $\{h_{s,t}^{-1}(e^{-iU_t})\}_{t\geq s}$ is the radial SLE$_6$ trace in the connected component of $\hat{\mathbb{C}}\setminus\overline{\mathbb{D}}$ that contains $\infty$. But $h_{s,t}^{-1}(e^{-iU_t})=g_s(g_t^{-1}(e^{-iU_t}))=g_s(\gamma(t))$, so the first part of the lemma follows.

For the second part of the lemma, let $\tilde{g}_s(z):=e^{-s}z$ and $g_{s,t}(z)=h_{s,t}(\tilde{g}_s(z))$. Then Proposition 3.30 of \cite{Law05a} implies
$$|g_t^{-1}(w)-g_{s,t}^{-1}(w)|=|g_s^{-1}(h_{s,t}^{-1}(w))-\tilde{g}_s^{-1}(h_{s,t}^{-1}(w))|\leq C e^s \text{ for any } |w|>1,$$
where $C>0$ is independent of $s, t$ and $w$. Let $w\rightarrow e^{-iU_t}$, we get
$$|\gamma(t)-\tilde{g}_s^{-1}(g_s(\gamma(t)))|\leq Ce^s.$$
Note $\tilde{g}_s^{-1}(g_s(\gamma(t)))$ has the same distribution as $\gamma^s$. Since $g_s(z)\sim e^{-s}z$ as $z\rightarrow\infty$, the Koebe $1/4$ Theorem (see Corollary 3.19 of \cite{Law05a}) implies $\gamma[-\infty,s]\subseteq\{z:|z|\leq 4 e^s\}$. Therefore the second part of the lemma follows.
\end{proof}
\begin{remark}
By the strong Markov property of Brownian motion, the first part of the lemma holds if $s$ is replaced by some stopping time of $\gamma$.
\end{remark}
Let $\gamma(t), -\infty<t<\infty$ be the trace of the full-plane SLE$_6$ in $\mathbb{C}$ from $0$ to $\infty$. Let $K_t$ be the hull generated by $\gamma[-\infty,t]$, i.e., the complement of the unbounded component of $\mathbb{C}\setminus\gamma[-\infty,t]$. Let $W_t, t\geq 0$ be a complex Brownian motion starting at the origin, and let $\hat{K_t}$ be the hull generated by $W[0,t]$. For any simply connected domain $D$ containing $0$, let $\sigma_D:=\inf\{t\geq -\infty:\gamma(t)\in\partial D\}$ and $\tau_D:=\inf\{t\geq 0: W_t\in\partial D\}$. Then Proposition 6.32 of \cite{Law05a} says that $K_{\sigma_{\mathbb{D}}}$ and $\hat{K}_{\sigma_{\mathbb{D}}}$ have the same distribution.

For any $\epsilon>0$, let $\gamma_{\delta}^{\epsilon}(t), 0\leq t\leq\infty$ be the radial exploration process from $\epsilon_{\delta}$ to $0_{\delta}$ in $(\epsilon\mathbb{D})_{\delta}$. Let $T_{\delta}^{\epsilon}:=\sup\{t\geq 0: \gamma_{\delta}^{\epsilon}(t)\in\partial (\epsilon\mathbb{D})_{\delta}\}$. Let $\gamma^{\epsilon}$ be the radial SLE$_6$ in $\epsilon\mathbb{D}$ from $\epsilon$ to $0$ and $T^{\epsilon}:=\sup\{t\geq0:\gamma^{\epsilon}(t)\in\partial(\epsilon\mathbb{D})\}$. Then the proof of Corollary \ref{cor} implies $\gamma_{\delta}^{\epsilon}[T_{\delta}^{\epsilon},\infty]$ converges weakly to $\gamma^{\epsilon}[T^{\epsilon},\infty]$. A little more work using Lemma 6.1 of \cite{CN06} and Lemmas 7.1 \& 7.2 of \cite{CN07} gives the boundary of the hull generated by $\gamma_{\delta}^{\epsilon}[T_{\delta}^{\epsilon},\infty]$, i.e., the complement of the unbounded component of $\overline{\mathcal{H}_{\delta}\setminus \Gamma(\gamma_{\delta}^{\epsilon}[T_{\delta}^{\epsilon},\infty])}$, converges weakly to the boundary of the hull generated by $\gamma^{\epsilon}[T^{\epsilon},\infty]$. Lemma \ref{lem3.1}, Corollary \ref{cor} and the proof of Proposition 6.32 of \cite{Law05a} imply the hull generated by $\gamma^{\epsilon}[T^{\epsilon},\infty]$ has the distribution of $\hat{K}_{\sigma_{\epsilon\mathbb{D}}}$. Applying Lemma \ref{lem3.1} again, we have
\begin{lemma}\label{lem5.2}
For any $\epsilon>0$, let $\beta_{\delta}^{\epsilon}(t), 0\leq t\leq\infty$ be the radial exploration process from $0_{\delta}$ to $\epsilon_{\delta}$ in $(\epsilon\mathbb{D})_{\delta}$. Let $S_{\delta}^{\epsilon}:=\inf\{t\geq 0: \beta_{\delta}^{\epsilon}(t)\in\partial (\epsilon\mathbb{D})_{\delta}\}$. Let $\beta^{\epsilon}$ be the time-reversal of $\gamma^{\epsilon}[T^{\epsilon},\infty]$. Then we have $\beta_{\delta}^{\epsilon}[0,S_{\delta}^{\epsilon}]$ converges weakly to $\beta^{\epsilon}$, and the hull generated by $\beta_{\delta}^{\epsilon}[0,S_{\delta}^{\epsilon}]$ converges weakly to $\hat{K}_{\sigma_{\epsilon\mathbb{D}}}$ as $\delta\downarrow 0$.
\end{lemma}
Now we have all ingredients to prove Theorem \ref{thm2}.
\begin{proof}[Proof of Theorem \ref{thm2}]
Let $\gamma_{\delta}(t), 0\leq t<\infty$ be the full plane exploration process in $\mathcal{H}_{\delta}$ from $0_{\delta}$ to $\infty$. Let $\gamma(t), -\infty\leq t\leq \infty$ be the full-plane SLE$_6$ in $\mathbb{C}$ from $0$ to $\infty$. For any $\epsilon>0$, we define $\tau_{\delta}^{\epsilon}:=\inf\{t\geq 0:\gamma_{\delta}(t)\in \partial(\epsilon\mathbb{D})_{\delta}\}$. It is clear $\gamma_{\delta}[0,\tau_{\delta}^{\epsilon}]$ has the same distribution as $\beta_{\delta}^{\epsilon}[0,S_{\delta}^{\epsilon}]$ (see Lemma \ref{lem5.2}). Let $K_{\delta}^{\epsilon}$ be the hull generated by $\gamma_{\delta}[0,\tau_{\delta}^{\epsilon}]$, i.e., the complement of the unbounded component of $\overline{\mathcal{H}_{\delta}\setminus \Gamma(\gamma_{\delta}[0,\tau_{\delta}^{\epsilon}])}$. Then Lemma \ref{lem5.2} implies $K_{\delta}^{\epsilon}$ converges weakly to $\hat{K}_{\sigma_{\epsilon\mathbb{D}}}$. Note that $\gamma_{\delta}[\tau_{\delta}^{\epsilon},\infty)$ is a radial exploration process in the unbounded component of $\overline{\mathcal{H}_{\delta}\setminus \Gamma(\gamma_{\delta}[0,\tau_{\delta}^{\epsilon}])}$. So the remark after Proposition \ref{prop5.1} implies $\gamma_{\delta}[\tau_{\delta}^{\epsilon},\infty)$ converges weakly to a radial SLE$_6$ $\tilde{\gamma}^{\epsilon}$ in $\mathbb{C}\setminus \hat{K}_{\sigma_{\epsilon\mathbb{D}}}$ aiming at $\infty$. Clearly, we have
$$\rho(\gamma_{\delta}[0,\infty],\gamma[-\infty,\infty])\leq \rho(\gamma_{\delta}[0,\infty],\gamma_{\delta}[\tau_{\delta}^{\epsilon},\infty])+\rho(\gamma_{\delta}[\tau_{\delta}^{\epsilon},\infty],\tilde{\gamma}^{\epsilon})+
\rho(\tilde{\gamma}^{\epsilon},\gamma[-\infty,\infty]).$$
The first term on the left hand side of the above inequality is bounded by $C\epsilon$ where $C$ comes from the equivalence of Euclidean metric and spherical metric in bounded region and thus $C$ is independent of $\epsilon$ and $\delta$ if $\epsilon<1$; the third term is also bounded by $C\epsilon$ by the remark after Lemma \ref{lem5.1} and the discussion right after that remark; the second term can be made arbitrarily small if $\delta$ is small by the discussion before the inequality. This completes the proof of Theorem \ref{thm2}.
\end{proof}

\begin{proof}[Proof of Theorem \ref{thm3}]
This is an immediate consequence of Theorem \ref{thm2} and its proof.
\end{proof}

\section*{Acknowledgements}
The author would like to thank Tom Kennedy for introducing him to this area of research and for many stimulating and helpful discussions. The author also thanks the referees for many valuable suggestions and comments.

\end{document}